\documentclass[a4paper]{article}
\usepackage{amsmath, amsthm, amssymb}
\usepackage[latin1]{inputenc}
\usepackage[T1]{fontenc}
\usepackage[english]{babel}     % the language of the document
\usepackage{color}
\usepackage{tikz}                      % to produce figures using tikz
\usetikzlibrary{snakes}
\usetikzlibrary{intersections}
\usepackage{pgf}
\usepackage{xkeyval}

\usepackage{multirow}           % for spanning rows in tabular environment

\usepackage[dvips]{epsfig}
\usepackage{color}
\usepackage{psfrag}
\usepackage[update,prepend]{epstopdf}

\usepackage{amsrefs}
\newtheorem{Proposition}{Proposition}[section]

\newtheorem{Theorem}{Theorem}[section]
\newtheorem{Remark}{Remark}[section]
\newtheorem{Example}{Example}[section]

\numberwithin{equation}{section}
\newcommand\be{\begin{eqnarray*}}
\newcommand\ee{\end{eqnarray*}}
\newcommand\ben{\begin{eqnarray}}
\newcommand\een{\end{eqnarray}}
\newcommand{\comment}[1]{}

\def\dvg{{\rm div}}

\def\IntO{\int\limits_\Omega}

\def\dx{\, {\rm d}\mathbf{x}}
\def\dxone{\, {\rm d}x}
\def\idata{\Upsilon}
\def\idatum{\gamma}
\def\solmap{\mathcal{S}}

\def\Eind{E^{\rm ind}}
\def\Eworst{E^{\rm worst}}
\def\Ebest{E^{\rm best}}
\def\Ecom{E^{\rm red}}

\title{Worst case approach in convex minimization problems
with uncertain data}
\author{O. Mali\footnote{University of Jyvaskyla, Department of Mathematical
Information Technology, P.O. Box 35 (Agora), FI-40014, University of Jyvaskyla,
Finland}}

\begin{document}

\maketitle

\begin{abstract}
This paper concerns quantitative analysis of errors generated by incompletely known data in convex minimization
problems.
The problems are discussed in the mixed setting and
the duality gap is used as the fundamental error measure. The influence of
the indeterminate data is measured using the worst case scenario
approach. The worst case error is decomposed into two computable quantities,
which allows the quantitative
comparison between errors resulting from the inaccuracy of the approximation
and the data uncertainty.
The proposed approach is demonstrated on a paradigm of a
nonlinear reaction-diffusion problem together with numerical examples.
%demonstrated
%for problems generated by power growth functionals.
%Results are confirmed by numerical tests.
\end{abstract}

\section{Introduction}

In mathematical models, the parameters (data) are often not completely known.
Instead, their values are based on measurements or assumptions.
For analysts, it is important to quantify the effect that the data uncertainty has on results generated
by the mathematical model. In particular, the data uncertainty
generates an accuracy limit, beyond which information cannot be extracted and all
(numerical or analytical) efforts become meaningless or even misleading.

The analysis of
uncertain data in PDEs (in the context discussed in this paper)
dates back to \cite{Babushka1961} (see also \cite{Elishakoff1983,Walsh1986}
and references therein). One of the main lines of the uncertainty analysis
is the probabilistic approach, where the
uncertain data is treated as a random variable. This leads to stochastic PDEs
(see, e.g., \cite{Schueller1997, Chow2007} and
references therein). A recent overview of the related numerical methods can be
found from \cite{Lordetall2014}. Analysis of
stochastic PDEs with additional noise can be found from
\cite{Helgeetall2010,Kovacsetall2012,Kovacsetall2013}.

The analysis in this paper is based on the duality theory (convex analysis) developed by
W. Fenchel, J.-J. Moreau,
and R. T. Rockafellar (see, e.g., \cite{EkelandTemam1976,Rockafellar1970}
and references therein).
S. Repin applied the duality theory
to derive computable a posteriori error estimates for convex problems
in \cite{Repin1997,Repin2000,NeittaanmakiRepin2004}.
In this paper, a class of convex variational problems (cf. (\ref{eq:Jdef:gen})) are
discussed and the quantity of interest is
the maximal duality gap related to a family of problems generated by the admissible
(incompletely known) data.
The approach is motivated by the worst
case scenario method (see, e.g., \cite{Hlavacek2004} and
\cite[Ch. 6.4]{BoydVanderberghe2004}). In contrast to the probabilistic methods,
the worst case scenario method does not assume which data are more or less
likely than other, but is solely interested to find the most drastic (measured in a
problem dependent manner) event possible. The benefit of the worst case scenario approach
is that there is no need to make any
assumptions concerning the probability distribution of unknown parameters, but
the results may be overpessimistic.
%
%Regardless of the
%fact that the presented analysis exploits the linearization w.r.t. incompletely
%known parameters (cf. (\ref{eq:Eexp})), the paper does
%not concern classical sensitivity analysis.
% (see, e.g.,
%\cite{MaslovPetrovsky1969,Cacuci1981a,Cacuci1981b,Cacuci2003}).

In this paper,
computable quantities (or two-sided estimates) related to the
approximation error and the error generated by
the incompletely known data are presented.
This is of practical importance. Firstly,
it is meaningless to use computational resources to improve approximations
which are ``too close'' to (or even inside) the solution set. Secondly,
it it crucial to evaluate whether the indeterminacy
of the data is small enough to justify any implications based on
the results (regardless of any numerical aspects).

For linear problems, the main ideas for the uncertainty analysis (related to
the one presented in this paper) were first discussed in
\cite[Ch. 6.8]{NeittaanmakiRepin2004} and \cite[Ch. 9.6]{Repin2008}.
These ideas rely on the use of functional a posteriori error estimates,
which are guaranteed and depend explicitly on the problem data.
Additional research is exposed in
\cite{MaliRepin2007,MaliRepin2010,MaliNeittaanmakiRepin2014,MaliRepin2015}.
It is worth mentioning that
the results partially coincide with the results of
\cite{BabushkaChatzipantelidis2002} obtained in a probabilistic framework
(see \cite[Rem. 3.1]{MaliRepin2015}).

The paper is composed as follows. In Sect. \ref{se:dual}, standard definitions and
results from the duality theory are promptly exposed. The relevant part for the
later analysis are, conjugate functionals, the definitions of
the primal and dual problems, the duality gap, and the necessary conditions
(subdifferential inclusions) in terms of the compound functionals (see, e.g., \cite{EkelandTemam1976}).
The use of the duality gap as the fundamental error quantity is motivated in
Sect. \ref{se:exdual}, where it is shown that
in linear examples (generated by quadratic energy functionals) certain important
a posteriori error estimation results, e.g., the Prager-Synge estimate \cite{PragerSynge1947},
are representations of the duality gap.

Analysis of the respective mathematical models, where the data are not completely known
is discussed in Sect. \ref{se:ind}. The general ideas based on the duality theory
 are explained in Sect. \ref{se:wc}. In Sect.
\ref{se:ex:wc}, the new error quantification method is applied to a nonlinear
reaction diffusion problem. Two-sided bounds for the error generated by
the incompletely known data are
presented in Theorem \ref{th:bounds}. The practicality of these bounds
is demonstrated in numerical examples in Sect. \ref{se:num}.

\section{Convex variational problems}
\label{se:dual}

\subsection{Definitions}

Let $V$ and $Y$ be a Banach spaces with $V^*$ and $Y^*$ as the corresponding dual spaces.
The duality pairing is denoted by brackets and subindices, e.g., $\langle v^*, v \rangle_V$,
where $v^* \in V^*$ and $v \in V$.
The set of functions $V \rightarrow \overline{\mathbb{R}}$, which
are pointwise supremum over a set of affine functions is denoted
by $\Gamma(V)$. $\Gamma_0(V)$ is the subset of functions other
than constants $+\infty$ and $-\infty$.
%\cite[Def.3.1]{EkelandTemam1976}.
Henceforth, assume that $F \in \Gamma_0(V)$ and $G \in \Gamma_0(Y)$.

The problems discussed in this paper are of the following type.
The primal problem is
\begin{equation} \label{eq:Jdef:gen}
J(u) = \inf\limits_{v \in V} J(v) := F(v) + G(\Lambda v) ,
\end{equation}
and the respective dual problem is
\begin{equation} \label{eq:Idef:gen}
I^*(p^*) = \sup\limits_{y^* \in Y^*} I^*(y^*) :=  -F^*(\Lambda^* y^*) - G^*(-y^*) ,
\end{equation}
where $F^*$ and $G^*$ are the respective Fenchel conjugates
\[
F^*(v^*) := \sup\limits_{v \in V} \{ \langle v^* , v \rangle_V - F(v) \}
\quad {\rm and} \quad
G^*(y^*) := \sup\limits_{y \in Y} \{ \langle y^* , y \rangle_Y - G(y) \},
\]
and
$\Lambda \in \mathcal L( V , Y )$ is a bounded linear operator with an adjoint
$\Lambda^* \in \mathcal L( Y^* , V^* )$, i.e.,
\begin{equation} \label{eq:adj}
\langle y^*, \Lambda v \rangle_Y = \langle \Lambda^* y^*, v \rangle_V ,
\quad \forall v \in V, \, y^* \in Y^* .
\end{equation}
The solutions of the primal and dual problem are characterized by the following Theorem
%characterizes the solutions
%$u \in V$ and $p^* \in Y^*$ of problems (\ref{eq:Jdef:gen}) and
%(\ref{eq:Idef:gen}), respectively
(see, e.g.,
\cite[Rem 4.2]{EkelandTemam1976} or \cite[Thr 2.2]{Repin2000}).
% i.e, the slopes of continuous affine minorants, which
%are exact at $v$.
%
\begin{Theorem} \label{th:ex}
Let
%$F \in \Gamma_0(V)$, $G \in \Gamma_0(Y)$, and
$J$ be coercive on $V$. If
there exists $u_0 \in V$ such that \mbox{$F(u_0) < +\infty$}, \mbox{$G(\Lambda u_0) < +\infty$}, and
$G$ is continuous at $\Lambda u_0$, then there exists
$u \in V$ and $p^* \in Y^*$, which are solutions of problems (\ref{eq:Jdef:gen}) and
(\ref{eq:Idef:gen}),
respectively. Moreover, the duality gap (denoted by $E(u,p^*)$) vanishes, i.e.,
\begin{equation} \label{eq:energ:equal}
E(u,p^*) := J(u) - I^*(p^*) = 0,
\end{equation}
and
\begin{align}
\label{eq:sub1}
\Lambda^* p^* & \in \partial F(u) ,\\
\label{eq:sub2}
- p^* & \in \partial G(\Lambda u) .
\end{align}
\end{Theorem}

In the problems (\ref{eq:Jdef:gen}) and
(\ref{eq:Idef:gen}),
the most fundamental quantities are the primal and dual energies ($J$ and $I^*$).
If an approximation \mbox{$(v,y^*) \in V \times Y^*$} is inaccurate in the
sense of these energies,
then it is clear that drawing any conclusions based on it is highly dubious.
The approximation error is naturally
measured by the value of the respective duality gap $E(v,y^*)$.

The duality gap can be studied in greater detail based on how well
the approximation satisfies the subgradient conditions (\ref{eq:sub1}) and (\ref{eq:sub2}).
%where, e.g., $\partial F(v) \subset V^*$ denotes the set of subgradients of
%$F$ at $v$.
The following Proposition (see, e.g., \cite[Prop. 5.1]{EkelandTemam1976}) characterizes the set of subgradients.
% \cite{EkelandTemam1976}.
%
\begin{Proposition} \label{pr:compound}
Under the earlier assumptions, $v^* \in \partial F(v)$ if and only if
\[
D_F(v,v^*) := F(v) + F^*(v^*) - \langle v^*, v \rangle_V = 0 .
\]
\end{Proposition}
\noindent The functional $D_F: V \times V^* \rightarrow \mathbb{R}$ is called the compound functional.
The Fenchel-Young inequality
\begin{equation*} %\label{eq:FY}
D_F(v,v^*) \geq 0,
\quad \forall v \in V, \, v^* \in V^*
\end{equation*}
and Proposition \ref{pr:compound} show that $D_F(v,v^*)$ is a reasonable measure of
the violation of the relation \mbox{$v^* \in \partial F(v)$}
(\cite[Thr. 2.2]{Repin2000} and \cite[(7.1.3)]{NeittaanmakiRepin2004}).

\begin{Remark}
The duality gap (\ref{eq:energ:equal}) can be decomposed as follows:
\begin{multline} \label{eq:gapdef}
E(v,y^*) = F(v) + G(\Lambda v) + F^*(\Lambda^* y^*) + G^*(-y^*) \\
= F(v) + G(\Lambda v) + F^*(\Lambda^* y^*) + G^*(-y^*)
- \langle - y^*, \Lambda v  \rangle_Y - \langle \Lambda^* y^*, v \rangle_V \\
= D_F(v,\Lambda^* y^*) + D_G(\Lambda v, -y^*) ,
\end{multline}
%\end{Remark}
%
where the two compound functionals measure
violations of relations (\ref{eq:sub1}) and (\ref{eq:sub2}), respectively. For analyst,
the form in (\ref{eq:gapdef}) is more practical than (\ref{eq:energ:equal}), since the values of the primal and
dual energies attained at the exact solution are unknown, but the compound functional values are always positive
and vanish only at the exact solution.
\end{Remark}

\subsection{Examples}
\label{se:exdual}

\subsubsection{Linear elliptic problem}
\label{se:lin}

Let $V$ and $Y$ be Hilbert spaces endowed with the inner products $(\cdot,\cdot)_V$
and $(\cdot,\cdot)_U$, and norms $\| \cdot \|_V$
and $\| \cdot \|_U$, respectively. Operators $\mathcal A:Y \rightarrow Y$ and $\mathcal B:V \rightarrow V$
are linear, bounded, symmetric, and strictly positive definite. Let $\ell \in V^*$.
This example concerns the linear problem
\begin{equation} \label{eq:ex:linlow}
\left( \Lambda^* \mathcal A \Lambda + \mathcal B \right) u = \ell .
\end{equation}
The primal energy functional associated with (\ref{eq:ex:linlow}) is
\[
J(v) = \tfrac{1}{2} \| \Lambda v \|_{\mathcal A}^2 +
\tfrac{1}{2} \| v \|_{\mathcal B}^2 - \langle \ell, v \rangle_V ,
%J(v) = \tfrac{1}{2} ( \mathcal A \Lambda v , \Lambda v )_Y +
%\tfrac{1}{2} ( \mathcal B v , v )_V - (  f, v )_V ,
\]
where
\[
\| y \|_{\mathcal A} := \sqrt{(\mathcal A y, y)_Y}
\quad {\rm and} \quad
\| v \|_{\mathcal B} := \sqrt{(\mathcal B v, v)_V}
\]
are equivalent norms in $Y$ and $V$, respectively.
To present $J(v)$ in the form (\ref{eq:Jdef:gen}), we set
\[
F(v) := \tfrac{1}{2} \| v \|_{\mathcal B}^2 - \langle \ell, v \rangle_V
\quad {\rm and} \quad
G(y) := \tfrac{1}{2} \| y \|_{\mathcal A}^2 .
\]
The corresponding conjugate functionals are
\[
%F^*(v^*) = \tfrac{1}{2} \|   (v^*+\ell) \|_{\mathcal B^{-1}}^2,
F^*(v^*) = \tfrac{1}{2} \| v^*+\ell \|_{\mathcal B^{-1}}^2
%F^*(v^*) = \tfrac{1}{2} ( \mathcal B^{-1} (f-v^*) , (f - v^*) )_{V^*}
\quad {\rm and} \quad
%G^*(y^*) = \tfrac{1}{2} \|   y^* \|_{\mathcal A^{-1}}^2 ,
G^*(y^*) = \tfrac{1}{2} \| y^* \|_{\mathcal A^{-1}}^2 ,
\]
where
\[
\| y \|_{\mathcal A^{-1}} := \sqrt{(\mathcal A^{-1} y, y)_Y}
\quad {\rm and} \quad
\| v \|_{\mathcal B^{-1}} := \sqrt{(\mathcal B^{-1} v, v)_V} .
\]
\begin{Remark}
Formally, one should write
\[
F^*(v^*) = \tfrac{1}{2} \| \mathcal I_V^{-1}  (v^*+\ell) \|_{\mathcal B^{-1}}^2
\quad
{\rm and}
\quad
G^*(y^*) = \tfrac{1}{2} \| \mathcal I_Y^{-1}  y^* \|_{\mathcal A^{-1}}^2,
\]
where $\mathcal I_V$ and $\mathcal I_Y$ denote canonical isomorphisms
(mapping of the Riesz representation
theorem for Hilbert spaces) $V \rightarrow V^*$ and $Y \rightarrow Y^*$.
Hereafter, we omit this detail  to lighten the notation.
\end{Remark}
The compound functionals associated with relations
(\ref{eq:sub1}) and (\ref{eq:sub2}) are
\begin{multline} \label{eq:linell:Df}
D_F(v,\Lambda^* y^*) = \tfrac{1}{2} \| v \|_{\mathcal B}^2
- \langle \ell, v \rangle_V
+ \tfrac{1}{2} \| \Lambda^* y^*+\ell \|_{\mathcal B^{-1}}^2
- \langle v , \Lambda^* y^* \rangle_V \\
= \tfrac{1}{2} \| \mathcal B v - \ell - \Lambda^* y^* \|_{\mathcal B^{-1}}^2
\end{multline}
and
\begin{multline} \label{eq:linell:Dg}
D_G(\Lambda v, -y^*) = \tfrac{1}{2} \| \Lambda v \|_{\mathcal A}^2
+ \tfrac{1}{2} \|   y^* \|_{\mathcal A^{-1}}^2
- \langle \Lambda v , -y^* \rangle_Y \\
= \tfrac{1}{2} \| \mathcal A \Lambda v +   y^* \|_{\mathcal A^{-1}}^2 ,
\end{multline}
respectively.
The fact that the solution $(u,p^*)$ satisfies
$D_F(u,\Lambda^* p^*) = 0$ and \linebreak $D_G(\Lambda u, -p^*) = 0$ leads to the
decomposed form of (\ref{eq:ex:linlow})
\begin{align}
\label{eq:ex:linlow:equi}
- \Lambda p^* + \mathcal B u & = \ell , \\
\label{eq:ex:linlow:dual}
p & = - \mathcal A \Lambda u .
\end{align}

\begin{Remark}
The compound functionals are squared (weighted)
$L^2$-norms of the residuals related to the mixed form
of the problem. Minimizing the duality gap w.r.t. primal and dual
variables over some finite dimensional subspaces leads to, e.g.,
least square finite element methods
(see, e.g., \cite{BochevGunzburger2009} and references therein).
\end{Remark}

\begin{Remark}
Substituting conditions (\ref{eq:ex:linlow:equi}) and
(\ref{eq:ex:linlow:dual}) to $D_F(v,\Lambda^* y^*)$
and $D_G(v,\Lambda^* y^*)$, respectively, yields
\[
D_F(v,\Lambda^* y^*) =
\tfrac{1}{2} \| \mathcal B(v - u) +  \Lambda^* (p^* - y^*) \|_{\mathcal B^{-1}}^2
\]
and
\[
D_G(\Lambda v,-y^*) =
\tfrac{1}{2} \| \mathcal A \Lambda(v - u) + y^* - p^* \|_{\mathcal A^{-1}}^2 .
\]
Expanding the terms and summing provides an alternative representation for
the duality gap (cross-terms cancel each other due to
(\ref{eq:adj}) and (\ref{eq:ex:linlow:dual}))
\begin{multline*} %\label{eq:ex:gapdef}
E(v,y^*) =
\tfrac{1}{2} \| v - u \|_{\mathcal B}^2
+ \tfrac{1}{2} \|   \Lambda^* (p^* - y^*) \|_{\mathcal B^{-1}}^2
+ \tfrac{1}{2} \| \Lambda(v - u) \|_{\mathcal A}^2
+ \tfrac{1}{2} \| y^*-p^* \|_{\mathcal A^{-1}}^2 .
\end{multline*}
By (\ref{eq:gapdef}), (\ref{eq:linell:Df}), (\ref{eq:linell:Dg}),
and (\ref{eq:ex:linlow:dual}) one obtains the so called
error equality
\begin{multline*}
\| v - u \|_{\mathcal B}^2
+
\|   \Lambda^* (p^* - y^*) \|_{\mathcal B^{-1}}^2
+
\| \Lambda(v - u) \|_{\mathcal A}^2
+
\| y^*-p^* \|_{\mathcal A^{-1}}^2 \\
=
\| \mathcal B v - \ell - \Lambda^* y^* \|_{\mathcal B^{-1}}^2 +
\| \mathcal A \Lambda v +   y^* \|_{\mathcal A^{-1}}^2 ,
\end{multline*}
which has importance in a posteriori error control and
adaptive finite element methods (see
\cite{CaiZhang2010,AnjamPauly2014}).
\end{Remark}

If $\mathcal B = 0$, then (\ref{eq:ex:linlow}) reduces to
\begin{equation} \label{eq:lin:nolow}
\Lambda^* \mathcal A \Lambda u = \ell .
\end{equation}
The primal energy functional is
\[
J(v) = \tfrac{1}{2} \| \Lambda v \|_{\mathcal A}^2 - \langle \ell, v \rangle_V ,
\]
where
\[
F(v) := - \langle \ell, v \rangle_V
\quad {\rm and} \quad
G(y) := \tfrac{1}{2} \| y \|_{\mathcal A}^2.
\]
The corresponding conjugate functionals are
\[
F^*(v^*) = \left\{ \begin{array}{ll}
0, & \quad {\rm if} \; v^* + \ell = 0, \\
+\infty, & \quad {\rm else} ,
\end{array}
\right.
\quad {\rm and} \quad
G^*(y^*) := \tfrac{1}{2} \|   y^* \|_{\mathcal A^{-1}}^2 .
\]
The respective compound functionals are
\[
D_F(v,\Lambda^* y^*) =
\left\{ \begin{array}{ll}
0 , & \quad {\rm if} \; \Lambda^* y^* + \ell =0, \\
+\infty, & \quad {\rm else}
\end{array}
\right.
\]
and
\[
D_G^*(\Lambda v, -y^*) := \tfrac{1}{2} \| \Lambda v \|_{\mathcal A}^2 +
\tfrac{1}{2} \|   y^* \|_{\mathcal A^{-1}}^2
+ \langle y^* ,\Lambda v  \rangle_Y
= \tfrac{1}{2} \| \mathcal A \Lambda v +   y^* \|_{\mathcal A^{-1}}^2 ,
\]
which clearly correspond to the equilibrium condition
$
\Lambda^* p^* + \ell = 0
$
and the duality condition
$
p^* + \mathcal A \Lambda u = 0
$,
respectively.

\begin{Remark}
Assume that $y^*$ is equilibrated, i.e., $\Lambda^* y^* + \ell = 0$. Then,
the duality gap is simply $D_G^*(\Lambda v, -y^*)$.
On the other hand, the duality condition for $(u,p^*)$ provides (again, by (\ref{eq:adj}))
\begin{multline*}
\| \mathcal A \Lambda v -   y^* \|_{\mathcal A^{-1}}^2
=
\| \mathcal A \Lambda v - \mathcal A \Lambda u + p^* - y^* \|_{\mathcal A^{-1}}^2 \\
=
\| \Lambda (v - u ) \|_{\mathcal A}^2 + \| p^* - y^* \|_{\mathcal A^{-1}}^2
+ 2 \langle \Lambda (v - u ) , p^* - y^* \rangle_Y ,
\end{multline*}
where the last term vanishes, since $y^*$ is equilibrated. This leads to the well known
Prager-Synge estimate \cite{PragerSynge1947} (for equilibrated fluxes)
\[
\| \Lambda (v - u ) \|_{\mathcal A}^2 + \| p^* - y^* \|_{\mathcal A^{-1}}^2
=
\| \mathcal A \Lambda v -   y^* \|_{\mathcal A^{-1}}^2 .
\]
\end{Remark}

\subsubsection{The power growth functional with lower order terms}
\label{se:pow}

Let
$V = W_0^{1,p'}(\Omega,\mathbb{R})$, $V^* = W^{-1,q'}(\Omega,\mathbb{R})$,
$Y=L^p(\Omega,\mathbb{R}^d)$, $Y^*=L^q(\Omega,\mathbb{R}^d)$,
$\Lambda = \nabla$, and $\Lambda^* = -\dvg$. Let $p' > 1$ and $p > 1$,
$q'$ and $q$ are the conjugate numbers defined by
$\tfrac{1}{p'}+\tfrac{1}{q'}=1$ and $\tfrac{1}{p}+\tfrac{1}{q}=1$, respectively.
The domain $\Omega \subset \mathbb{R}^d$ is open, bounded and has a piecewise Lipschitz continuous boundary.
Moreover, let $p \leq p'$, then $L^{p'}(\Omega,\mathbb{R}) \subset L^{p}(\Omega,\mathbb{R})$ (see, e.g.,
\cite[Th. 2]{Villani1985}) and
$W_0^{1,p'}(\Omega,\mathbb{R}) \subset L^{p}(\Omega,\mathbb{R})$.
This example concerns a $p$-Poisson problem with lower order terms, i.e.,
\begin{align*}
-\dvg \left( \gamma' | \nabla u|^{p'-2} \nabla u \right)
+ \gamma |u|^{p-2} u & = f ,
\quad \textrm{in } \Omega \\
u & = 0 ,
\quad \textrm{on } \partial \Omega .
\end{align*}
These Euler's equations are related to the primal energy functional
\begin{equation} \label{eq:pow:J}
J(v) = \IntO \left( \tfrac{1}{p'} \gamma' | \nabla v |^{p'} +
\tfrac{1}{p} \gamma |v|^p -  fv \right) \dx ,
\end{equation}
where $f \in L^{q}(\Omega,\mathbb{R})$ and $\gamma, \gamma' \in L^\infty(\Omega,\mathbb{R}_{+})$. Here,
\[
F(v) =  \IntO \left( \tfrac{1}{p} \gamma |v|^p - fv \right) \dx
\quad \textrm{and} \quad
G(\mathbf{y}) = \tfrac{1}{p'} \IntO  \gamma' |\mathbf{y}|^{p'}  \dx .
\]
The corresponding conjugates are
\[
F^*(v^*) = \tfrac{1}{q} \IntO \left( \tfrac{1}{\gamma} \right)^{q-1} |v^* + f|^q \dx
\quad \textrm{and} \quad
G^*(\mathbf{y}^*) = \tfrac{1}{q'} \IntO \left( \tfrac{1}{\gamma'} \right)^{q'-1}
|\mathbf{y}^*|^{q'} \dx ,
\]
and the compound functionals are (assume henceforth additional regularity
$\dvg \mathbf{y}^* \in L^q(\Omega,\mathbb{R})$).
\[
D_F(v,-\dvg \mathbf{y}^*) =
\IntO \left(
\tfrac{1}{p} \gamma |v|^p +
\tfrac{1}{q} \left( \tfrac{1}{\gamma} \right)^{q-1} |-\dvg \mathbf{y}^* + f|^q -
(-\dvg \mathbf{y}^* + f) v  \right) \dx
\]
and
\[
D_G(\nabla v, -\mathbf{y}^*) =
\IntO \left(
\tfrac{1}{p'} \gamma' |\nabla v|^{p'} +
\tfrac{1}{q'} \left( \tfrac{1}{\gamma'} \right)^{q'-1} | \mathbf{y}^* |^{q'} -
\nabla v \cdot \left(-\mathbf{y}^*\right) \right) \dx .
\]

\begin{Remark}
Note that in the case $p=p'=2$ (then also $q=q'=2$)
\[
D_F(v,-\dvg \mathbf{y}^*)  = \tfrac{1}{2} \IntO \tfrac{1}{\gamma}
\left( \gamma v + \dvg \mathbf{y}^* - f \right)^2 \dx
\]
and
\[
D_G(\nabla v, -\mathbf{y}^*) = \tfrac{1}{2} \IntO \tfrac{1}{\gamma'}
\left( \gamma' \nabla v + \mathbf{y}^* \right)^2 \dx .
\]
Assuming enough regularity, conditions \mbox{$D_F(u,-\dvg \mathbf{p}^*)=0$}
and \mbox{$D_G(\nabla u, -\mathbf{p}^*)=0$}
lead to
\begin{alignat*}{2}
\dvg \mathbf{p}^* + \gamma u & = f , && \quad \textrm{in } \Omega, \\
\mathbf{p}^* & = - \gamma' \nabla u , && \quad \textrm{in } \Omega, \\
u & = 0 , && \quad \textrm{on } \partial \Omega .
\end{alignat*}
\end{Remark}

\section{Incompletely known data}
\label{se:ind}

\subsection{Worst case scenario approach}
\label{se:wc}

In many mathematical models
functionals $F$ and/or $G$ depend
on some incompletely known data or parameters. Henceforth the set of these
parameters is denoted by
$\idata$.
%Here, the treatment of incompletely known data is
%deterministic and related to the worst case scenario analysis (see, e.g., \cite{Hlavacek2004}).
Motivated by the engineering practise, we assume that the set of admissible data is
of the form
\begin{equation*} %\label{eq:datadef:gen}
\idata := \{ \idatum = \idatum_\circ + \delta ,
\; \| \delta \|_\idata \leq \varepsilon \},
\end{equation*}
where $\idatum_\circ$ are the known ``mean data'' and $\delta$ are the respective
(bounded by some suitable norm $\| \cdot \|_\idata$) variations resulting, e.g.,
from the measurement inaccuracy,
incomplete known material parameters, etc.
It is assumed that problems (\ref{eq:Jdef:gen}) and
(\ref{eq:Idef:gen}) have solutions for all admissible data. This relation is denoted by
\[
\solmap : \idata \rightarrow \solmap(\idata) \subset V \times Y^* .
\]

Note that every $\idatum \in \idata$ generates not only a problem and the
associated solution $\solmap(\idatum)$, but also the proper error measure
(i.e., the duality gap)
$E_\idatum$ (subscript denotes the dependence of functionals on the parameter $\idatum$)
associated with that problem.
The fact that the error measures are not uniquely
defined stimulates the following definitions of
extremal errors associated with an arbitrary \mbox{$(v,y^*) \in V \times Y^*$},
\begin{equation*}
\Eworst(v,y^*) := \sup\limits_{\idatum \in \idata} E_\idatum(v,y^*)
\end{equation*}
and
\begin{equation*}
\Ebest(v,y^*) := \inf\limits_{\idatum \in \idata} E_\idatum(v,y^*) .
\end{equation*}

\begin{Remark}
Recall that by (\ref{eq:gapdef}),
the duality gap is a sum of compound functionals
\[
E_\idatum(v,y^*) = D_{F_\idatum}(v,\Lambda^* y^*) + D_{G_\idatum}(\Lambda v, -y^*) .
\]
Since compound functionals vanish at the exact solution (and are positive otherwise), they
could have been used independently to define similar worst (and best)
case errors.
\end{Remark}

Due to the computational cost, typically
simulations are performed using only the mean data values.
Thus the objective of computations is to approximate  the mean solution
$(u_\circ,p^*_\circ) := \solmap(\idatum_\circ)$.
In the context of the mixed setting discussed in this paper, the approximation
is ``good'' if $E_{\idatum_\circ}(v,y^*)$ is small. (Hereafter, we denote
$E_{\idatum_\circ}$ by $E_\circ$ to avoid double subscripts.)
However, if the data is not completely known, then there is an accuracy limit
beyond which further computations (e.g., adaptive mesh refinements) make no
sense. The following analysis provides tools to compare between the
error resulting from the inaccuracy of the approximation and the error
generated by the data uncertainty.

Assume that it is possible to explicitly linearize the duality gap functional
with respect to the data perturbation
\begin{equation} \label{eq:Eexp}
E_\idatum(v,y^*) = E_\circ(v,y^*) + R(v,y^*;\idatum_\circ,\delta) ,
\end{equation}
where $R$ is the remainder term related to
the particular mean data $\idatum_\circ$ and perturbation $\delta$.
Then $\Eworst$ can be decomposed by (\ref{eq:Eexp}) into two parts
\begin{equation} \label{eq:Ehcomp}
\Eworst(v,y^*) =
E_\circ(v,y^*)
 +
\Eind(v,y^*),
\end{equation}
where
\begin{equation} \label{eq:Eind:def}
\Eind(v,y^*) := \sup\limits_{\| \delta \|_\idata \leq \varepsilon}
R(v,y^*;\idatum_\circ,\delta) .
\end{equation}
In (\ref{eq:Ehcomp}), $E_\circ(v,y^*)$ is the approximation error and
$\Eind(v,y^*)$ is the error generated by the incompletely known data.
The main task is to compute (or estimate) the respective supremum in (\ref{eq:Eind:def}).
The difficulty of this task depends on the problem and the set of possible
data perturbations.

The described error decomposition is schematically illustrated for a case where
$E_\circ$ dominates $\Eind$ and
vice versa in Fig. \ref{fig:Ecomp}. In Fig. \ref{fig:Ecomp} (top), the approximation
$(v,y^*)$ is so far from the exact mean solution $(u_\circ,p_\circ^*)$
that all admissible error measures yield relatively similar value. In this case, it would be
reasonable to improve the approximation, e.g., by refining the mesh. In Fig. \ref{fig:Ecomp}
(bottom), $(v,y^*)$ is so close to $(u_\circ,p_\circ^*)$ (and the entire set
$\solmap(\idata)$) that different admissible duality gap functionals provide
very different values. This means that the uncertainty in the data dominates and reducing the
error w.r.t. one error measure could result in increasing it in another. Since all
error measures are equally possible, improving the approximation is impossible due to
the accuracy limit.

For numerical analyst,
the decomposition (\ref{eq:Ehcomp}) gives means to compare errors generated by different
sources. The comparison between different error sources can be done, e.g., by
computing ratios
\begin{equation} \label{eq:ratios}
\rho^{\rm approx} := \frac{E_\circ}{\Eworst}
\quad \textrm{and} \quad
\rho^{\rm indet} := \frac{\Eind}{\Eworst} .
\end{equation}

%The error decomposition can be used, e.g., for
% benchmark computations, where one computes very high accuracy approximations and studies
%solely the error due to indeterminacy. It is worth noting that any conforming
%method can be used to generate conforming approximations.
%
%
%%%%%%%%%%%%%%%%%%%%%%%%%%%%%%%%%%%%%%%%%%%%%%%%%%%%%%%%%%%%%%%%%%%%%%%%%%%%%
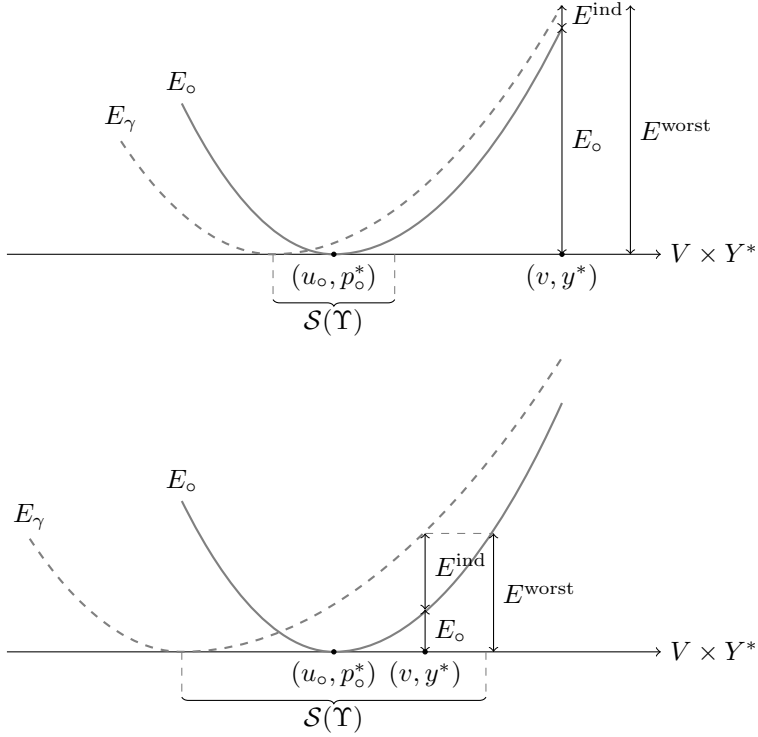
\begin{figure}[tb]
\begin{center}
\begin{tikzpicture}[scale=1]
\def\axlen{4.3};

% shifting parameters for text
\def\lowker{0.6};
\def\lowkeri{0.9};

% solution set radius
\def\solset{0.8};

% parabola parameters
\def\Eoleft{2.0};
\def\Eoright{3.0};
\def\Eolefth{2.0};
\def\Eorighth{3.0};

\def\Eileft{2.0-\solset};
\def\Eiright{3.0};
\def\Eilefth{1.5};
\def\Eirighth{3.3};

\def\Eiileft{1.0+\solset};
\def\Eiiright{1.8+\solset};
\def\Eiilefth{1.3};
\def\Eiirighth{1.8};

% approximation location
\def\app{3.0};

% axis
\draw[thin, ->] (-\axlen,0)--(\axlen,0);
\draw (\axlen,0) node[anchor=west] {$V \times Y^*$};

% duality gap E_0
\draw[thick,gray=80!] (-\Eoleft, \Eolefth) parabola bend (0,0) (\Eoright, \Eorighth);
%\draw (-\Eoleft, \Eolefth) node[anchor=south] {$E_\circ(v,y^*)$};
\draw (-\Eoleft, \Eolefth) node[anchor=south] {$E_\circ$};

% duality gap E_i
\draw[thick,dashed,gray=90!]
(-\Eileft, \Eilefth) parabola bend (-\solset,0) (\Eiright, \Eirighth);
%\draw (-\Eileft, \Eilefth) node[anchor=south] {$E_\idatum(v,y^*)$};
\draw (-\Eileft, \Eilefth) node[anchor=south] {$E_\idatum$};

% duality gap E_ii
%\draw[thick,dotted,gray=95!]
%(-\Eiileft, \Eiilefth) parabola bend (\solset,0) (\Eiiright, \Eiirighth);
%\draw (-\Eiileft, \Eiilefth) node[anchor=south] {$E_\idatum(v,y^*)$};

% solution
\draw[fill=black] (0,0) circle (0.8pt);
\draw (0,0) node[anchor=north] {$(u_\circ,p_\circ^*)$};

% solution set
\draw[thin, dashed, gray=90!] (-\solset,0) -- (-\solset,-\lowker);
\draw[thin, dashed, gray=90!] (\solset,0) -- (\solset,-\lowker);
\draw[snake=brace, mirror snake] (-\solset,-\lowker) -- (\solset,-\lowker);
\draw (0,-\lowker) node[anchor=north] {$\solmap(\idata)$};

% tilde Emax
%\def\Eindii{0.6};
%\draw[thin, <->] (0,0) -- (0,\Eindii);
%\draw (0,\Eindii/2+0.05) node[anchor=east] {$\Eind$};

% approximation
\draw[fill=black] (\app,0) circle (0.8pt);
\draw (\app,0) node[anchor=north] {$(v,y^*)$};

% errors of approximation
\draw[thin, <->] (\app,0) -- (\app,\Eorighth);
\draw[thin, <->] (\app,\Eorighth) -- (\app,\Eirighth);
\draw (\app,\Eorighth/2) node[anchor=west] {$E_\circ$};
\draw (\app,0.2+\Eorighth) node[anchor=west] {$\Eind$};

%
%\draw[thin, dashed, gray=90!] (\app,\Eorighth)-- (\app+\lowkeri,\Eorighth);
%\draw[thin, dashed, gray=90!] (\app,\Eirighth)-- (\app+\lowkeri,\Eirighth);
\draw[thin, <->] (\app+\lowkeri,0) -- (\app+\lowkeri,\Eirighth);
\draw (\app+\lowkeri,\Eirighth/2) node[anchor=west] {$\Eworst$};

%\end{tikzpicture}
%\caption{Error decomposition for a coarse approximation}.
%\label{fig:Ecomp}
%\end{center}
%\end{figure}
%
%%%%%%%%%%%%%%%%%%%%%%%%%%%%%%%%%%%%%%%%%%%%%%%%%%%%%%%%%%%%%%%%%%%%%%%%%%%%%%%%%
% ACCURATE APPROXIMATION
%
\begin{scope}[yshift=-150]
\def\axlen{4.3};

% shifting parameters for text
\def\lowker{0.6};
\def\lowkeri{0.9};

% solution set radius
\def\solset{2.0};

% parabola parameters
\def\Eoleft{2.0};
\def\Eoright{3.0};
\def\Eolefth{2.0};
\def\Eorighth{3.3};

\def\Eileft{2.0-\solset};
\def\Eiright{3.0};
\def\Eilefth{1.5};
\def\Eirighth{3.9};

\def\Eiileft{1.0+\solset};
\def\Eiiright{1.8+\solset};
\def\Eiilefth{1.3};
\def\Eiirighth{1.8};

% approximation location
\def\app{1.2};

% axis
\draw[thin, ->] (-\axlen,0)--(\axlen,0);
\draw (\axlen,0) node[anchor=west] {$V \times Y^*$};

% duality gap E_0
\draw[thick,gray=80!] (-\Eoleft, \Eolefth) parabola bend (0,0) (\Eoright, \Eorighth);
%\draw (-\Eoleft, \Eolefth) node[anchor=south] {$E_\circ(v,y^*)$};
\draw (-\Eoleft, \Eolefth) node[anchor=south] {$E_\circ$};

% duality gap E_i
\draw[thick,dashed,gray=90!]
(-\Eileft, \Eilefth) parabola bend (-\solset,0) (\Eiright, \Eirighth);
%\draw (-\Eileft, \Eilefth) node[anchor=south] {$E_\idatum(v,y^*)$};
\draw (-\Eileft, \Eilefth) node[anchor=south] {$E_\idatum$};

% solution
\draw[fill=black] (0,0) circle (0.8pt);
\draw (0,0) node[anchor=north] {$(u_\circ,p_\circ^*)$};

% solution set
\draw[thin, dashed, gray=90!] (-\solset,0) -- (-\solset,-\lowker);
\draw[thin, dashed, gray=90!] (\solset,0) -- (\solset,-\lowker);
\draw[snake=brace, mirror snake] (-\solset,-\lowker) -- (\solset,-\lowker);
\draw (0,-\lowker) node[anchor=north] {$\solmap(\idata)$};

% tilde Emax
%
\def\Eindi{1.57}
\def\Eindii{0.56};
%\draw[thin, <->] (0,0) -- (0,\Eindii);
%\draw (0,\Eindii/2+0.05) node[anchor=east] {$\Eind$};

% approximation
\draw[fill=black] (\app,0) circle (0.8pt);
\draw (\app,0) node[anchor=north] {$(v,y^*)$};

% errors of approximation
\draw[thin, <->] (\app,0) -- (\app,\Eindii);
\draw[thin, <->] (\app,\Eindii) -- (\app,\Eindi);
\draw (\app,\Eindii/2) node[anchor=west] {$E_\circ$};
\draw (\app,0.6+\Eindii) node[anchor=west] {$\Eind$};

%
%\draw[thin, dashed, gray=90!] (\app,\Eorighth)-- (\app+\lowkeri,\Eorighth);
\draw[thin, dashed, gray=90!] (\app,\Eindi)-- (\app+\lowkeri,\Eindi);
\draw[thin, <->] (\app+\lowkeri,0) -- (\app+\lowkeri,\Eindi);
\draw (\app+\lowkeri,\Eindi/2) node[anchor=west] {$\Eworst$};
\end{scope}

\end{tikzpicture}
\caption{Error decomposition for a coarse (top) and an accurate approximation (bottom).}
\label{fig:Ecomp}
\end{center}
\end{figure}
%
%%%%%%%%%%%%%%%%%%%%%%%%%%%%%%%%%%%%%%%%%%%%%%%%%%%%%%%%%%%%%%%%%%%%%%%%%%%%%%%%%%%%%%%

\begin{Remark}
In simulation practise, $\Ebest$ is typically of lesser importance than $\Eworst$.
For the sake of completeness, we apply the decomposition (\ref{eq:Ehcomp}) to $\Ebest$ and
obtain (analogously to (\ref{eq:Eexp}) and (\ref{eq:Eind:def}))
\[
\Ebest(v,y^*) = E_\circ(v,y^*) + \Ecom(v,y^*),
\]
where
\[
\Ecom(v,y^*) := \inf\limits_{\| \delta \|_\idata \leq \varepsilon}
R(v,y^*;\idatum_\circ,\delta)
\]
is an error reduction term. It adjusts the value of the approximation
error in such a way that
the value of the error becomes the most
favorable one (among the possible error measures).
Analogously $E_\circ$
and $\Eind$, the comparison between $E_\circ$
and $\Ecom$ also provides information on how close the approximation is to the solution
set. The approximation is coarse, if $E_\circ \gg |\Ecom|$ and it is accurate (close
to the solution set), if $E_\circ \approx - \Ecom$.
\end{Remark}

\begin{Remark} \label{re:diam}
Note that $\Eworst$ is always strictly positive, if the set of admissible data is greater
than a single value. On the other hand, $\Ebest(v,y^*)$ is zero if (and only if)
$(v,y^*)$ belongs to the solution set, i.e.,
$(v,y^*) \in \solmap(\idata)$. In a certain sense,
$\Eworst$ and $\Ebest(v,y^*)$ characterize ``the diameter of the solution set''
\[
{\rm diam}_{(v,y^*)}(\solmap(\idata)) := \Eworst(v,y^*) - \Ebest(v,y^*)
\]
measured ``from the direction of $(v,y^*)$''. In the special case
$(v,y^*) \in \solmap(\idata)$,
${\rm diam}_{(v,y^*)}(\solmap(\idata))$ resembles more ``a radius of the solution set''.
\end{Remark}

\begin{Remark}
In contrast to Remark \ref{re:diam}, it is clear that a real characteristic
measure of the solution set should be
independent of any approximation. It should depend only on the problem data and
the data indeterminacy. There are two natural alternatives for such a measure,
\[
\Eworst(u_\circ,p_\circ^*)
\quad
\textrm{and}
\quad
%\widetilde \Eind(u_\circ,p_\circ^*) :=
\sup\limits_{\idatum \in \idata} E_\circ(\solmap(\idatum)) .
\]
These two quantities answer to questions
``How badly our mean solution may
violate the necessary conditions related to the family of admissible problems?'' and
``How badly an arbitrary admissible solution may violate the
necessary conditions of the mean problem?'', respectively.
In linear problems, the duality gap is associated with the mixed norm
(see Sect. \ref{se:lin}) and the latter quantity has been successfully
estimated in \cite{MaliRepin2015} for problems of the form (\ref{eq:lin:nolow}).

\end{Remark}

%\subsection{Linear reaction diffusion problem}
%\label{se:ex:lin:wc}
%

\subsection{Nonlinear reaction diffusion problem}
\label{se:ex:wc}

Consider the problem defined in
Sect. \ref{se:pow}, where $p'=2$, $2>p>1$, and  $\gamma'=1$. Then, $V = H^1_0(\Omega)$,
\mbox{$V^*=H^{-1}(\Omega)$}, and
\[
J(v) = \IntO \left( \tfrac{1}{2} | \nabla v |^{2} +
\tfrac{1}{p} \gamma |v|^p -  fv \right) \dx .
\]
If the solution is regular enough, then the problem can be written as a stationary
reaction diffusion problem with a nonlinear reaction term.
\begin{align}
\label{eq:pow:nec}
- \Delta u + \gamma |u|^{p-2} u & = f, \quad {\rm in} \; \Omega , \\
\nonumber
            u & = 0, \quad {\rm on} \; \partial \Omega .
\end{align}

Assume that $\gamma$ is not fully known, but belongs to a set
\begin{equation} \label{eq:datadef}
\gamma \in \idata := \{ \gamma \in L^\infty(\Omega,\mathbb{R}) \, | \,
\gamma = \gamma_\circ + \delta \} ,
\end{equation}
where $\gamma_\circ \in L^\infty(\Omega,\mathbb{R}_{+})$ is the known ``mean value'' and
$\delta \in L^\infty(\Omega,\mathbb{R})$ is the perturbation.
The condition
\begin{equation} \label{eq:deltabound}
\| \delta \|_{L^\infty(\Omega,\mathbb{R})} \leq \varepsilon <
 \min\limits_{\mathbf{x} \in \Omega} \gamma_\circ(\mathbf{x}) .
\end{equation}
guarantees that all elements of $\idata$ are strictly positive.
The compound functionals (generating the respective duality gap by (\ref{eq:gapdef})) for the problem are
\begin{equation} \label{eq:Dfpow}
D_F(v,-\dvg \mathbf{y}^*) =
\IntO \left(
\tfrac{1}{p} \gamma |v|^p +
\tfrac{1}{q} \left( \tfrac{1}{\gamma} \right)^{q-1} \! \! |- \dvg \mathbf{y}^* + f|^q -
(- \dvg \mathbf{y}^* + f) v  \right) \dx
\end{equation}
and
\[
D_G(\nabla v, -\mathbf{y}^*) = \tfrac{1}{2} \| \nabla v + \mathbf{y}^* \|_{L^2(\Omega,\mathbb{R}^d)}^2 .
\]
The following Theorem provides two-sided computable bounds for the indeterminacy related
error.
\begin{Theorem} \label{th:bounds}
Let $\Eind(v,\mathbf{y}^*)$ be as in (\ref{eq:Ehcomp}),
where $\idatum$ is defined by (\ref{eq:datadef}) and
(\ref{eq:deltabound}), and
$(v,\mathbf{y}^*) \in H^1_0(\Omega) \times Q$, where
\[
Q := \{ \boldsymbol{\eta}^* \in L^q(\Omega,\mathbb{R}^d)
\, | \, \dvg \boldsymbol{\eta}^* \in L^q(\Omega,\mathbb{R}) \} .
\]
Then,
\[
\underline{\Eind}(v,\mathbf{y}^*) \leq \Eind(v,\mathbf{y}^*) \leq
\overline{\Eind}(v,\mathbf{y}^*) ,
\]
where
\begin{align}
\label{eq:th1:low}
\underline{\Eind}(v,\mathbf{y}^*) & :=
\IntO \max\{ h(-\varepsilon,v,\mathbf{y}^*,\varepsilon) , h(\varepsilon,v,\mathbf{y}^*,\varepsilon) \} \dx , \\
\label{eq:th1:up}
\overline{\Eind}(v,\mathbf{y}^*) & :=
\IntO \max\{ h(-\varepsilon,v,\mathbf{y}^*,-\varepsilon) , h(\varepsilon,v,\mathbf{y}^*,-\varepsilon) \} \} \dx ,
\end{align}
and
\begin{multline*}
h(\delta,v,\mathbf{y}^*,c) :=
    \delta \left(
        \tfrac{1}{p} |v|^p + (1-q) \gamma_\circ^{-q}\tfrac{1}{q} |-\dvg \mathbf{y}^* + f|^q
    \right )   \\
+
    \delta^2 \left(
        \tfrac{1}{2} (q-1)(\gamma_\circ+c)^{-q-1} |-\dvg \mathbf{y}^* + f|^q
    \right) .
\end{multline*}
\end{Theorem}
\begin{proof}
The decomposition (\ref{eq:Eexp}) is derived by linearizing
the (nonlinear) term containing $\gamma$ in (\ref{eq:Dfpow}).
More precisely, the (first order) McLaurin series of
\mbox{$
z(\delta) := (\gamma_\circ+\delta)^{1-q}
$}
is of interest, i.e.,
\begin{align}
\nonumber
z(\delta) & = z(0) + z'(0) \delta + \tfrac{1}{2} z''(c) \delta^2 \\
\label{eq:th1:1}
\left( \gamma_\circ+\delta \right)^{1-q} & =
\gamma_\circ^{1-q} +
(1-q) \gamma_\circ^{-q} \delta -
\tfrac{1}{2} q(1-q)(\gamma_\circ+c)^{-q-1} \delta^2 ,
\end{align}
where $c \in (0,\delta)$ in the remainder term. Applying
(\ref{eq:th1:1}) into (\ref{eq:Dfpow})
($D_G$ does not depend on $\gamma$) leads to
\[
E_\gamma(v,\mathbf{y}^*) = E_\circ(v,\mathbf{y}^*) + R(v,\mathbf{y}^*;\gamma_\circ, \delta) ,
\]
where
\begin{multline*} %\label{eq:th1:2}
R(v,\mathbf{y}^*;\gamma_\circ, \delta) :=
\IntO \Big(
    \delta \left(
        \tfrac{1}{p} |v|^p + (1-q) \gamma_\circ^{-q}\tfrac{1}{q} |-\dvg \mathbf{y}^* + f|^q
    \right )  \\
-
     \delta^2
        \tfrac{1}{2} (1-q)(\gamma_\circ+c)^{-q-1} |-\dvg \mathbf{y}^* + f|^q
    \Big) \dx .
\end{multline*}

The remaining task is to estimate
$
\sup\limits_{\| \delta \|_{L^\infty(\Omega,\mathbb{R})} < \varepsilon} R(v,\mathbf{y}^*;\gamma_\circ, \delta)
$
from both sides.
This can be done pointwise. Define
\[
A(\mathbf{x}) := \tfrac{1}{p} |v(\mathbf{x})|^p +
\tfrac{1}{q} (1-q) \gamma_\circ(\mathbf{x})^{-q} |-\dvg \mathbf{y}^*(\mathbf{x}) + f(\mathbf{x})|^q
\]
and
\[
B(\mathbf{x};c) := -\tfrac{1}{2}(1-q)(\gamma_\circ(\mathbf{x})+c)^{-q-1}
|-\dvg \mathbf{y}^*(\mathbf{x}) + f(\mathbf{x})|^q .
\]
Then one can write
\[
R(v,\mathbf{y}^*;\gamma_\circ, \delta) =
\IntO
h(\delta,v,\mathbf{y}^*,c)
\dx
=
\IntO
\left( \delta(\mathbf{x})  A(\mathbf{x}) + \delta(\mathbf{x})^2 B(\mathbf{x};c) \right) \dx .
\]
Note that $B \geq 0$ a.e. $\mathbf{x}$ in $\Omega$, since $q>1$ and $\gamma_\circ+c>0$ by the fact that
$c \in (0,\delta)$ (or $c \in (\delta,0)$, if $\delta<0$) and (\ref{eq:deltabound}).
Thus the integrand is an upwards opening parabola w.r.t. $\delta$ (at each point)
and the maximal values are attained at the endpoints of the interval
$(-\varepsilon,\varepsilon)$. Moreover, the function
$c \rightarrow (\gamma_\circ(x)+c)^{-q-1}$ is
monotonically decreasing. Thus, the maximum of $B(\mathbf{x},c)$ w.r.t. $c$ can be estimated from
above by setting $c=-\varepsilon$ and from below by setting $c=\varepsilon$.
These observations lead to (\ref{eq:th1:low}) and (\ref{eq:th1:up}).
\end{proof}

\subsection{Numerical examples}
\label{se:num}

Consider the incompletely known reaction diffusion problem described in Sect.
\ref{se:ex:wc}, where $\Omega := (0,1)$. The corresponding primal energy functional is
\[
J(v) =
\int\limits_0^1 \left( \tfrac{1}{2} (v')^2 + \gamma \tfrac{1}{p} |v|^p - fv \right) \dxone
\]
and the respective dual energy functional is
\[
I^*(y^*)  =
-\int\limits_0^1 \left( \tfrac{1}{q} \left( \tfrac{1}{\gamma} \right)^{q-1} |-(y^*)'+f|^q
+ \tfrac{1}{2} |y^*|^2 \right) \dxone .
\]
The incompletely known function $\gamma$ is defined by (\ref{eq:datadef}).
The approximations $v_h$ and $y^*_h$ are computed by minimizing and maximizing
primal and dual energies (associated with the ``mean data'' $\gamma_\circ$) over
finite dimensional subspaces, respectively.
The minimization and maximization are done by Matlab
(version 2013a, \verb+fminunc+ -function, quasi-Newton method, BFGS-scheme
\cite{Broyden1970,Goldfarb1970,Fletcher1970,Shanno1970}).
The subspaces are
generated by the finite element method using standard Lagrange elements.
\begin{Remark}
Generally, there is no need to solve both the primal and dual problem.
If the applied solution method generates only an approximation of $u$,
then the approximation of $p^*$ can be typically obtained by postprocessing.
\end{Remark}
The accuracy of the pair $(v_h,y^*_h)$ is measured using the
compound functionals
\[
D_F(v,y') = \int\limits_0^1 \left(
\tfrac{1}{p} \gamma_\circ |v|^p +
\tfrac{1}{q} \left( \tfrac{1}{\gamma_\circ} \right)^{q-1} \! \! | -(y^*)' + f|^q -
(-(y^*)' + f) v  \right) \dxone
\]
and
\[
D_G(v',-y) = \tfrac{1}{2} \| v' + y^* \|_{L^2((0,1))}^2 ,
\]
and the respective duality gap.

\begin{Example} \label{ex:1}
Let $p=1.2$, $\gamma_\circ = 1$, and $f=1$. Approximations $v_h$ and
$y^*_h$ are computed using linear
Lagrange elements.
The convergence of the solution is shown in Fig. \ref{fig:ex1:conv}.
The
convergence of the duality gap together with the bounds for $\Eind$
(Theorem \ref{th:bounds}) with various levels of indeterminacy
in $\gamma$ (i.e., $\varepsilon$ value in
(\ref{eq:deltabound})) are depicted in Figure \ref{fig:ex1:bounds}.
The bounds for $\Eind$ can be easily used to generate
bounds for ratios $\rho^{\rm approx}$ and $\rho^{\rm indet}$ (defined in (\ref{eq:ratios})), which are shown in Figure
\ref{fig:ex1:ratios}.
If the function $\gamma$ is known by accuracy $\varepsilon = 0.005$, then refining the mesh
would be reasonable. However, if $\varepsilon = 0.01$, then all computations
with more than 60 DOFs are dubious. Moreover, if $\varepsilon = 0.05$, then already the
coarsest approximations with 11 DOFs is sufficient.
\end{Example}
\begin{figure}
\begin{center}
\includegraphics[width=0.9\textwidth]{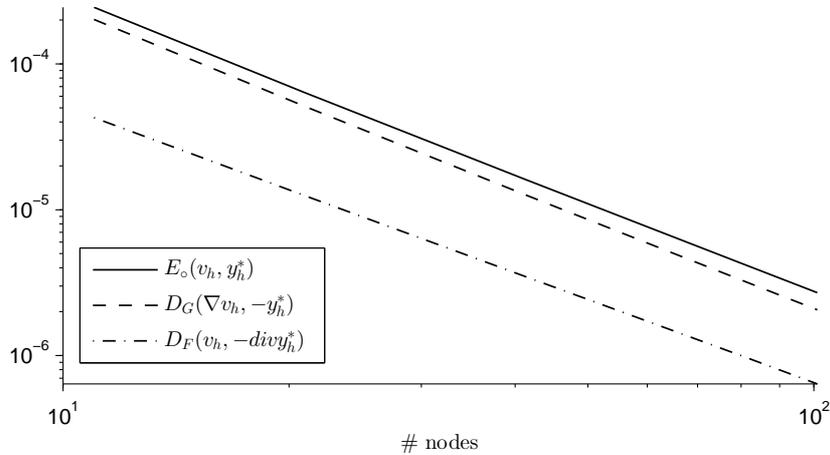}
\caption{Example \ref{ex:1}: Convergence of the approximations measured in terms of
compound functionals and the duality gap}
\label{fig:ex1:conv}
\end{center}
\end{figure}
\begin{figure}
\begin{center}
\includegraphics[width=0.9\textwidth]{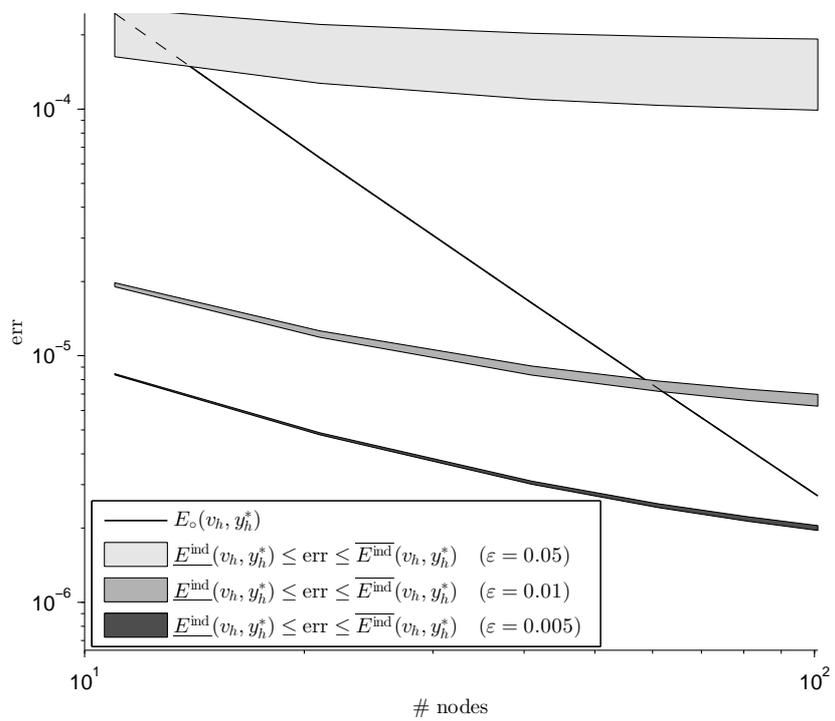}
\caption{Example \ref{ex:1}: Convergence of the approximation error in terms of the
duality gap and the error due to indeterminacy}
\label{fig:ex1:bounds}
\end{center}
\end{figure}
\begin{figure}
\begin{center}
\includegraphics[width=0.9\textwidth]{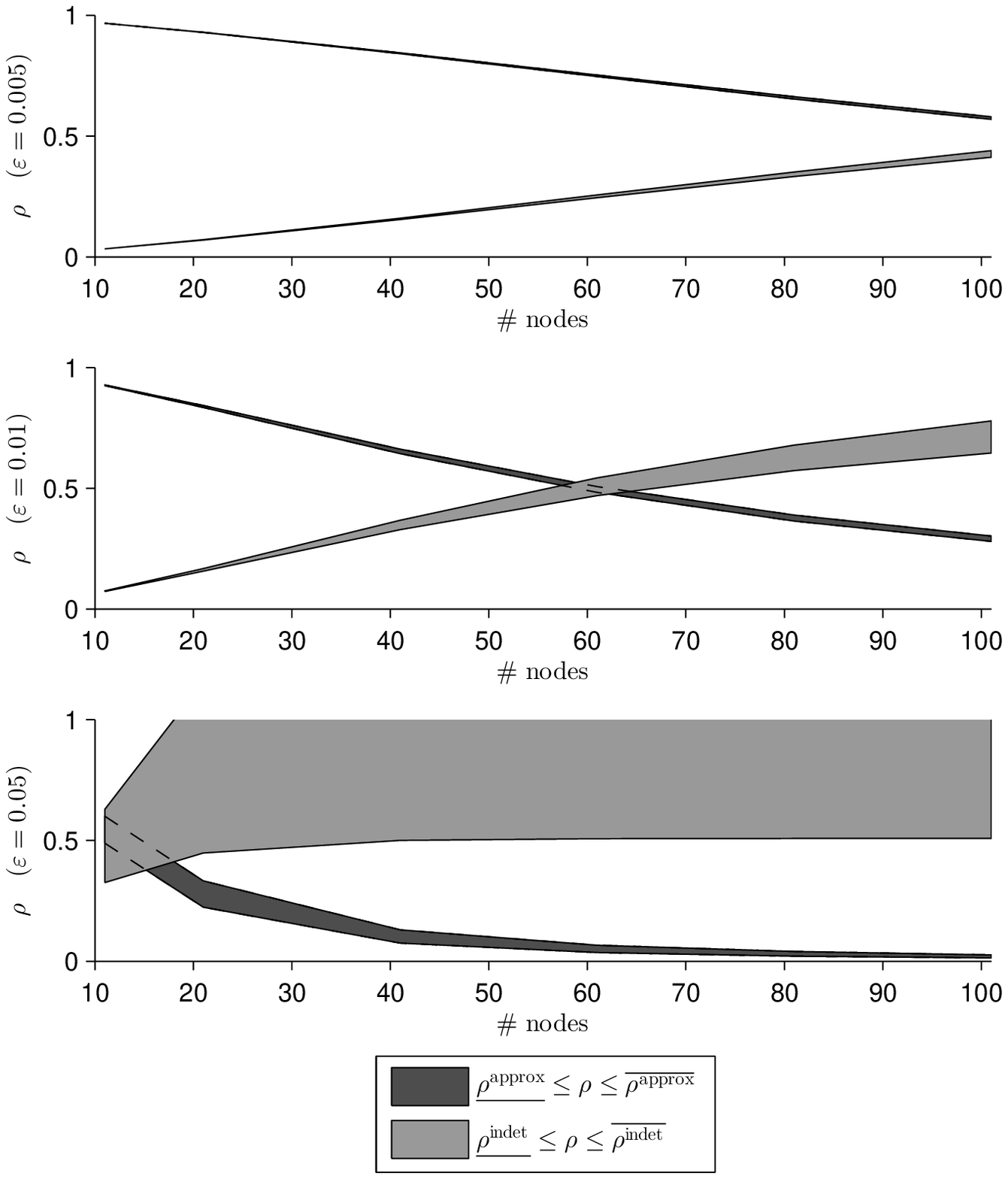}
\caption{Example \ref{ex:1}: Bounds for ratios $\rho^{\rm approx}$ and $\rho^{\rm indet}$
in (\ref{eq:ratios})}
\label{fig:ex1:ratios}
\end{center}
\end{figure}

In numerical practise, nonlinear problems have to be solved by some iterative method
(for suitable optimization methods, see, e.g.,
\cite{Bazaraaetall2006,BoydVanderberghe2004}).
Thus the approximation error is generated not only by the mesh, but also due to the
iteration. For analyst, it might be difficult to judge
whether one should refine the mesh or improve
the iterative solver in order to obtain more accurate approximations.
Moreover, computational cost of the iteration (e.g., optimization method)
increases quickly as the degrees of freedom are increased and computing
approximations on unreasonably refined meshes (w.r.t. data indeterminacy) should
be avoided.
 The following examples show how bounds from Theorem
\ref{th:bounds} can be applied to define
reasonable mesh size using a method of manufactured solutions.

\begin{Example} \label{ex:2}
Let $p=1.2$, $\gamma_\circ(x) = 1+x$, and  $u(x) = \sin(2 \pi x)$. Then
$p^*(x)= - u'(x)$ and the respective $f(x)$ is easy to compute by (\ref{eq:pow:nec}).
The pair $(I_h u, I_h p^*)$
denotes the interpolated solutions on a given mesh.
The convergence (using linear and quadratic Lagrange elements) of the
duality gap together with the bounds for $\Eind(I_h u, I_h p^*)$ with
different indeterminacy levels $\varepsilon$
are depicted in Fig. \ref{fig:ex2:bounds}.
It shows the efficiency of using quadratic elements and reasonable mesh sizes
for different indeterminacy levels $\varepsilon$ in (\ref{eq:deltabound}). In particular,
Fig. \ref{fig:ex2:bounds} indicates that the use of more than approximately 100 DOFs
(with quadratic elements) is unreasonable for the current problem setting unless the
data is known by accuracy much less than $\varepsilon = 0.005$.
\end{Example}
\begin{figure}
\begin{center}
\includegraphics[width=0.9\textwidth]{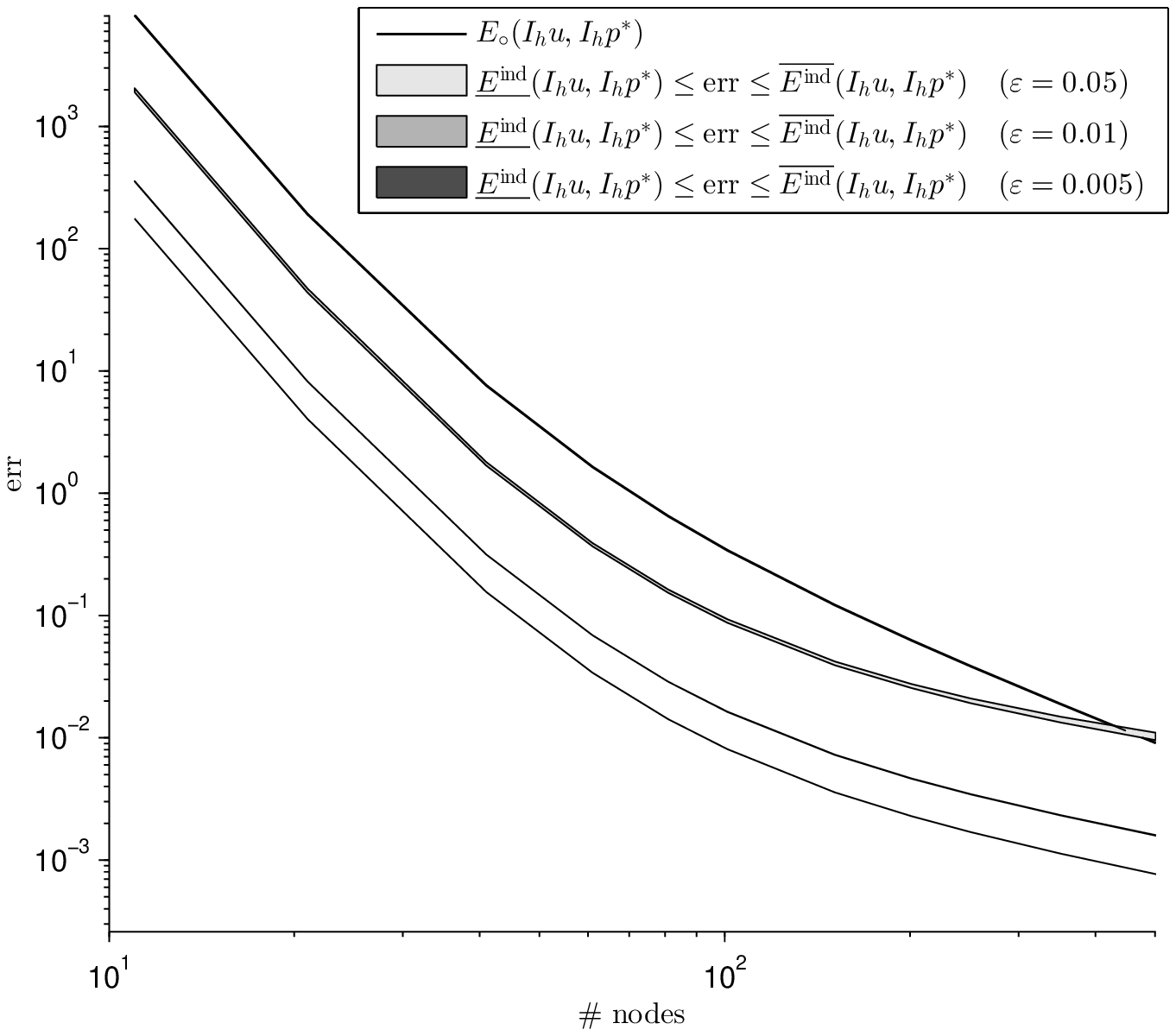} \\
\includegraphics[width=0.9\textwidth]{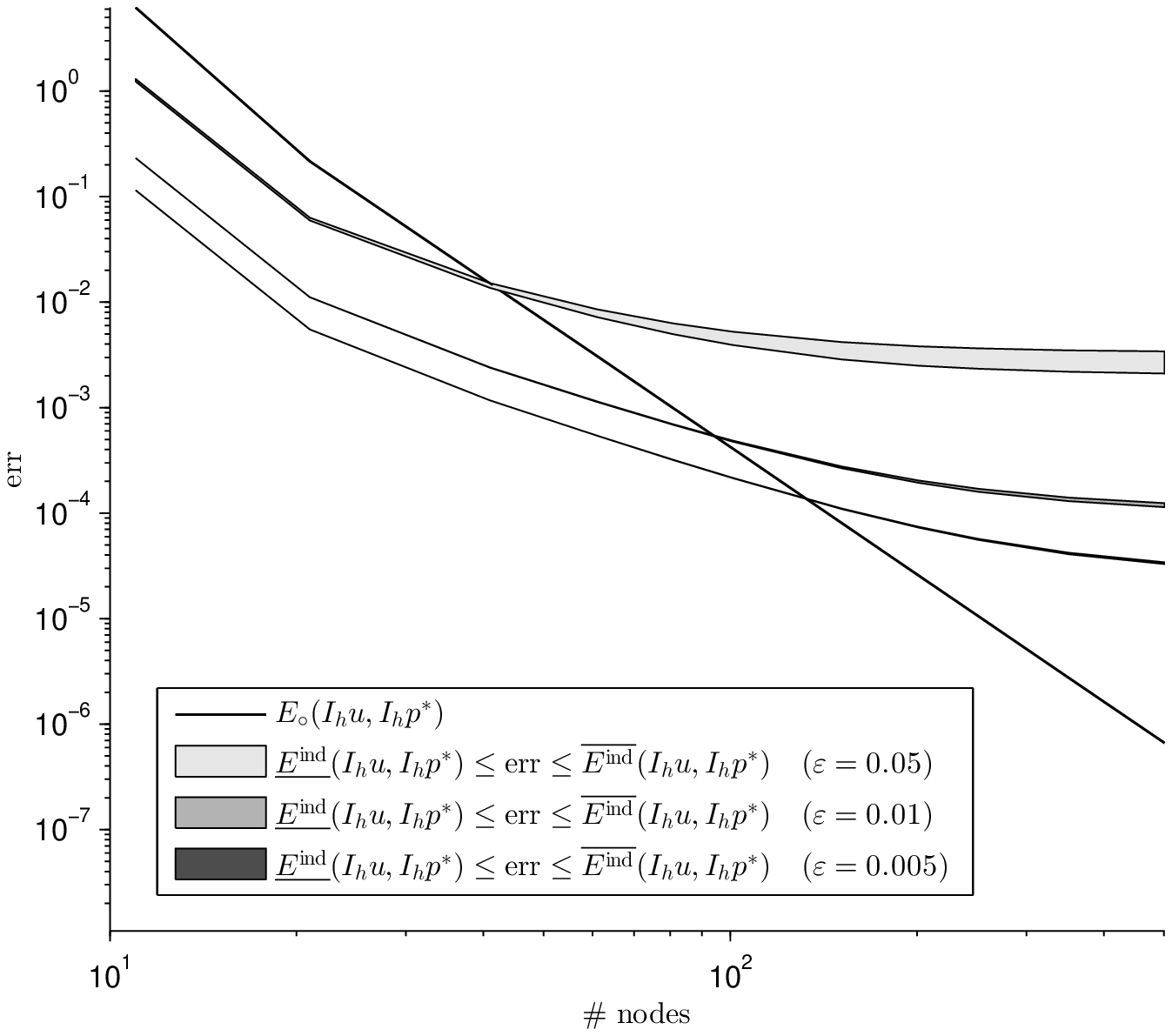} \\
\caption{Example \ref{ex:2}: Convergence of $E_\circ(I_h u, I_h p^*)$ for linear (above) and
quadratic (below) elements together with bounds for $\Eind(I_h u, I_h p^*)$ with
various indeterminacy levels $\varepsilon$ in (\ref{eq:deltabound})}
\label{fig:ex2:bounds}
\end{center}
\end{figure}

\begin{Example} \label{ex:3}
Repeat Example \ref{ex:2}, but set $u(x) = \sin(4 \pi x)$. This results in a
more oscillating $f$ in comparison with Example \ref{ex:2}.
The convergence (for quadratic elements) of the
duality gap together with the bounds for $\Eind(I_h u, I_h p^*)$ with
different indeterminacy levels $\varepsilon$
are depicted in Fig. \ref{fig:ex3:bounds}.
\end{Example}
\begin{figure}
\begin{center}
\includegraphics[width=0.9\textwidth]{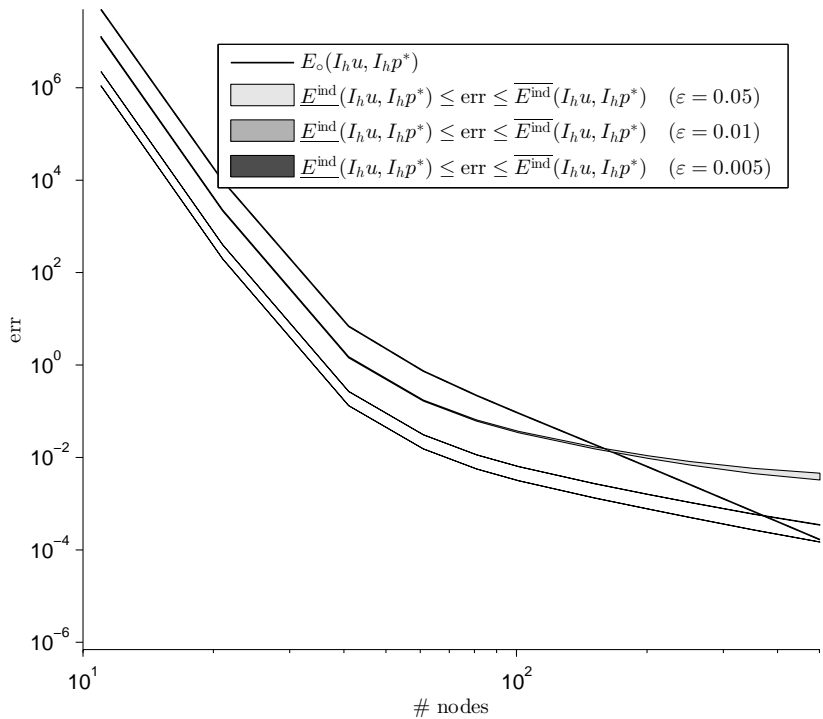}
\caption{Example \ref{ex:3}: Convergence of $E_\circ(I_h u, I_h p^*)$ for
quadratic elements together with bounds for $\Eind(I_h u, I_h p^*)$ with
various indeterminacy levels $\varepsilon$ in (\ref{eq:deltabound})}
\label{fig:ex3:bounds}
\end{center}
\end{figure}

The comparison between Examples \ref{ex:2} and \ref{ex:3} illustrates
that for a more complicated problem (highly oscillating righthand side)
the reasonable mesh refinement
level (in comparison with the data uncertainty) is much higher. In other words,
for a more complicated problem, the accuracy limit dictated by the
incompletely known data is not met without additional computational expenditure.
Note that this procedure could be done in a systematic way by precomputing
proper mesh refinement levels for a wide range of different known $u$
 (and the respective $f$) and indeterminacy levels.
Then an analyst could compare his or hers given $f$ and data accuracy
to the benchmark computations and define a reasonable characteristic mesh size.
The same information could be used as a stopping criteria for any adaptive
solver, since reasonable value for $E_\circ(v,y^*)$ would be known a priori.

It would be possible to study the model problem in great detail, e.g.,
the effect that the parameter $p$ has to the
indeterminacy error, compare the indeterminacy in the $\gamma'$ and $\gamma$ in
(\ref{eq:pow:J}), etc. However, it is not the aim of this paper to concentrate on
properties of this particular model problem, but to demonstrate a
convincing methodology that can be applied to study the effects of incompletely known
data for a wide range of convex minimization problems.

\bibliographystyle{plain}
%\bibliography{bib/lib}
%\bibliographystyle{alpha}
%\bibliography{bibliography}
\bibliography{../../../main_bibliography}

\end{document}